\documentclass{amsart}
\usepackage[margin=1in]{geometry}

\usepackage{color,amssymb, comment, mathrsfs,tikz,upgreek,contour, soul, colonequals, mathdots,fancyvrb}
\usepackage[colorlinks, linkcolor={blue}]{hyperref}
\usepackage{tikz-cd}
\usepackage{cleveref}
\usetikzlibrary{cd}
\usetikzlibrary{fit,shapes.geometric}
\usetikzlibrary{matrix}
\usetikzlibrary{arrows}
\usepackage{enumitem}
\usepackage{mathtools}
\usepackage{thmtools}
\usepackage[all]{xy}
\usepackage[mathlines,pagewise]{lineno}
\allowdisplaybreaks

\newtheorem{theorem}{Theorem}[section]
\newtheorem*{maintheorem1}{Main Theorem 1}
\newtheorem*{maintheorem2}{Main Theorem 2}
\newtheorem{lemma}[theorem]{Lemma}
\newtheorem{proposition}[theorem]{Proposition}
\newtheorem{corollary}[theorem]{Corollary}
\theoremstyle{definition}
\newtheorem*{ack}{Acknowledgements}

\usepackage[only,llbracket,rrbracket,llparenthesis,rrparenthesis]{stmaryrd} 
\usepackage{accsupp}

\theoremstyle{definition}

\newtheorem{example}[theorem]{Example}

\theoremstyle{remark}
\newtheorem{remark}[theorem]{Remark}

\numberwithin{equation}{section}

\newcommand{\kk}{\Bbbk}

\newcommand{\Ext}{\operatorname{Ext}}

\newcommand{\Hom}{\operatorname{Hom}}

\newcommand{\RHom}{\operatorname{\mathbf{R}Hom}}

\newcommand{\Ass}{\operatorname{Ass}}

\newcommand{\HH}[2]{\operatorname{H}_{#1}(#2)}
\newcommand{\Tor}{\operatorname{Tor}}

\newcommand{\m}{\mathfrak{m}}

\newcommand{\depth}{\operatorname{depth}}

\newcommand{\ld}{\operatorname{lc.dim}}

\newcommand{\p}{\mathfrak{p}}
\newcommand{\Spec}{\operatorname{Spec}}

\renewcommand{\H}{\mathrm{H}}

\newcommand{\lotimes}{\otimes^{\mathbf{L}}}

\newcommand{\catb}{\sqsubset\mspace{-13mu}\sqsupset}
\newcommand{\catbb}{\sqsupset}
\newcommand{\catba}{\sqsubset}
\newcommand{\Catsub}[3]{{\mathsf{#2}}_{#3}(#1)}
\newcommand{\Catsupsub}[4]{{\mathsf{#2}}^{\text{\upshape #3}}_{#4}(#1)}
\newcommand{\Catsup}[3]{{\mathsf{#2}}^{\text{\upshape #3}}(#1)}
\newcommand{\D}{\textsf{\upshape D}}
\newcommand{\Db}[1][R]{\Catsub{#1}{D}{\catb}}
\newcommand{\Dfb}[1][R]{\Catsupsub{#1}{D}{f}{\catb}}
\newcommand{\Dbb}[1][R]{\Catsub{#1}{D}{\catbb}}
\newcommand{\Dba}[1][R]{\Catsub{#1}{D}{\catba}}
\newcommand{\Dfba}[1][R]{\Catsupsub{#1}{D}{f}{\catba}}
\newcommand{\Dfbb}[1][R]{\Catsupsub{#1}{D}{f}{\catbb}}
\newcommand{\Df}[1][R]{\Catsup{#1}{D}{f}}

\newcommand{\LLam}[2][\mathfrak{m}]{\nobreak{\mathbf{L}\Lambda^{#1}#2}}
\newcommand{\LL}[3][\mathfrak{m}]{\nobreak{\mathbf{L}\Lambda^{#1}_{#2}#3}}
\newcommand{\RGam}[2][\mathfrak{m}]{\nobreak{\mathbf{R}\Gamma_{#1}#2}}
\newcommand{\RG}[3][\mathfrak{m}]{\nobreak{\mathbf{R}\Gamma_{#1}^{#2}#3}}
\newcommand{\amp}{\operatorname{amp}}
\newcommand{\projdim}{\operatorname{projdim}}
\newcommand{\flatdim}{\operatorname{flatdim}}
\newcommand{\injdim}{\operatorname{injdim}}

\newcommand{\HR}{\HH{0}{R}}

\makeatletter
\DeclareFontEncoding{LS2}{}{\@noaccents}
\makeatother
\DeclareFontSubstitution{LS2}{stix}{m}{n}

\DeclareSymbolFont{largesymbolsstix}{LS2}{stixex}{m}{n}

\DeclareMathDelimiter{\lbrbrak}{\mathopen}{largesymbolsstix}{"EE}{largesymbolsstix}{"14}
\DeclareMathDelimiter{\rbrbrak}{\mathclose}{largesymbolsstix}{"EF}{largesymbolsstix}{"15}

\keywords{DG-ring; Improved New Intersection Theorem; amplitude inequality; depth; local cohomology dimension; Cohen-Macaulay DG-ring}

\subjclass[2020]{13D09, 13D22, 16E35, 16E45, 16E65}

\author[Ferraro]{Luigi Ferraro}
\address[Luigi Ferraro]{School of Mathematical and Statistical Sciences \\ University of Texas Rio Grande Valley\\ Edinburg, TX 78539, USA}
\email{luigi.ferraro@utrgv.edu}

\author[Nason]{Zachary Nason}
\address[Zachary Nason]{Department of Mathematics\\ University of Nebraska-Lincoln\\
Lincoln, NE 68503, USA}
\email{znason2@huskers.unl.edu}

\begin{document}
\title{Intersection Theorems over DG-Rings Revisited}

\begin{abstract}
In this work we generalize two recently proved intersection theorems for DG-rings. The Derived Improved New Intersection Theorem concerns the length of semi-free DG-modules over DG-rings and it was recently proved by the second author. We show that it holds under weaker hypotheses. Foxby's Intersection Theorem was generalized to DG-rings by Yang and we improve the inequality that they provided. As an application we prove a DG version of the classic result that finite length modules of finite projective dimension only exist over Cohen-Macaulay rings, generalizing another result of Yang.
\end{abstract}

\maketitle

\section{Introduction}
In classical commutative algebra, the Improved New Intersection Theorem is a central result concerning the length of a complex
\[
F: 0 \to F_n \to \ldots \to F_1 \to F_0 \to 0
\]
with each $F_i$ a finitely generated free module over a noetherian local ring ($R$, $\m$). In the theorem's original statement, it asserts that if the homologies $\HH{i}{F}$ are of finite length for $i > 0$, and if a minimal generator of $\HH{0}{F}$ generates a submodule of finite length, then $n \geq \dim R$. Various authors have generalized the original statement of the Improved New Intersection Theorem, culminating in work by the first author and Christensen \cite{IntThm}. They proved that for a complex $F$ as described above, if an ideal $I \subseteq R$ annihilates $\HH{i}{F}$ for $i > 0$ as well as a nonzero minimal generator of $\HH{0}{F}$, then $n \geq \dim R - \dim R/I$ \cite[Theorem 2.2]{IntThm}.

In this paper, we prove a generalization of \cite[Theorem 2.2]{IntThm} for nonnegative commutative noetherian local differential graded (DG) rings with bounded homology. (We will formally define this class of DG-rings in the next section.) Such DG-rings are natural generalizations of the noetherian local rings studied in classical commutative algebra, and for simplicity we will from now on refer to them as ``DG-rings''. Various theorems comparable to the classical homological conjectures/theorems have been recently proven in the DG-ring setting. For example, Shaul \cite[Theorem 5.22]{ShaulCM} has proven a DG-version of Bass' Theorem concerning Cohen-Macaulay DG-rings, and Yang \cite[Theorem 3.1]{yang} has recently shown that a DG-version of Foxby's Intersection Theorem for DG-rings holds. In \cite[Theorem 5.6]{Zach}, the second author established a DG-version of the original statement of the Improved New Intersection Theorem by using the existence of big Cohen-Macaulay modules over noetherian local rings. The main result of our paper weakens the hypotheses of \cite[Theorem 5.6]{Zach}.

\begin{maintheorem1}[\Cref{thm:dinit}]
Let $R$ be a DG-ring with constant amplitude. Let $I$ be an ideal of $\H_0(R)$, and $F\in\Dfb[R]$ with $\inf F=0$. If $\H_i(F)$ is $I$-torsion for $i>0$ and a minimal generator of $\H_0(F)$ is $I$-torsion, then
\[
\projdim_RF+\amp R\geq\dim\H_0(R)-\dim\H_0(R)/I.
\]
\end{maintheorem1}
In particular, we allow the homologies of the graded-free DG $R$-module $F$ to be $I$-torsion, instead of requiring them to have finite length as it was done in \cite[Theorem 5.6]{Zach}.

We also establish an improvement on Yang's DG-version of Foxby's Intersection Theorem \cite[Theorem 3.1]{yang}. Similar to what Foxby originally proved in \cite{foxflat}, Yang shows that for a DG-ring $R$, the following inequality holds:

\[
\flatdim_R X+\ld X\lotimes_RY \geq \ld_RY
\]
where $X$ and $Y$ are DG $R$-modules such that $X \in \Db[R]$ with $\depth_R X < \infty, \flatdim_RX<\infty$ and $Y \in \Dfbb[R]$. We replaced $\flatdim_RX$ in the above inequality with the smaller quantity $\depth R - \depth_R X$ and replaced $\ld_RX \lotimes_R Y$ with the smaller quantity $-\inf\RGam[\m]{(X\lotimes_R Y)}$:

\begin{maintheorem2}[\Cref{thm:foxby2}]
Let $X\in\Db[R]$ and $Y\in\Dfbb[R]$ with $Y\not\simeq0$, $\flatdim_R X<\infty$ and $\depth_R X<\infty$. Then,
\[
\depth R-\depth_R X-\inf\RGam[\m]{(X\lotimes_R Y)}\geq\ld_R Y.
\]
\end{maintheorem2}
We also construct an explicit example of a DG-module over a DG-ring of positive amplitude that shows that the above inequality is stronger than Yang's inequality.

This inequality has applications for the amplitude of the local cohomology of a derived torsion product of DG $R$-modules (\Cref{cor:ampRGam}), from which we deduce a DG version of the classic result that finite length modules of finite projective dimension only exist over Cohen-Macaulay rings (\Cref{cor:fl}), generalizing a recent result of Yang \cite[Corollary 4.4]{yang}. Namely, we show that the DG-ring $R$ is Cohen-Macaulay if and only if there exists a bounded derived $\m$-torsion DG $R$-module $X$ of finite flat dimension with the same amplitude as $R$. 

We now outline the structure of this work. In \Cref{prelim} we provide the necessary background information. In \Cref{sec:dinit} we prove our strengthening of the Derived Improved New Intersection Theorem. In \Cref{sec:foxby} we prove our DG-version of Foxby's intersection theorem, generalizing \cite[Theorem 3.1]{yang}, and we provide some applications.
 
\begin{ack}
Luigi Ferraro was partly supported by the Simons Foundation grant MPS-TSM-00007849.  
\end{ack}

\section{Preliminaries}\label{prelim}

\subsection*{DG-Rings} A differential graded ring (DG-ring) is a graded algebra $R = \bigoplus_{i \in \mathbb{Z}} R_i$ equipped with a differential $\delta_i: R_i \to R_{i-1}$ such that $\delta_{i-1}\circ\delta_i = 0$ for all $i \in \mathbb{Z}$ and such that the differential satisfies the Leibniz rule
\begin{equation*}
\delta(rs) = \delta(r)s + (-1)^{|r|}r\delta(s)
\end{equation*}
where $|r|$ denotes the index such that $r \in R_{|r|}$. A DG-ring is commutative if for all elements $r, s \in R$, we have that $rs = (-1)^{|r||s|}sr$ and if for all elements $r \in R$ with $|r|$ odd, we have $r^2 = 0$.

The main object of study in this article is a noetherian local DG-ring. This is a commutative DG-ring that satisfies the four conditions below:
\begin{enumerate}
\item $R$ is non-negative ($R_i = 0$ for all $i < 0$)
\item $R$ is homologically bounded ($\HH{i}{R} = 0$ for $i \gg 0$.)
\item $\HR$ is a noetherian local ring with maximal ideal $\m$ and residue field $\kk$
\item $\HH{i}{R}$ is a finitely generated $\HR$-module for all $i \in \mathbb{Z}$ 
\end{enumerate}
Throughout the rest of the paper, we will refer to these objects as ``DG-rings'' for simplicity.

\subsection*{The Derived Category} A DG-module $M$ over a DG-ring $R$ is a graded module $M = \bigoplus_{i \in \mathbb{Z}}M_i$ equipped with a differential $\delta_i: M_i \to M_{i-1}$ such that $\delta_{i-1}\circ\delta_i = 0$ for all $i \in \mathbb{Z}$ and such that the differential satisfies the Leibniz rule
\begin{equation*}
\delta(rm) = \delta(r)m + (-1)^{|r|}r\delta(m)
\end{equation*}
where $r \in R$ and $m \in M$. The set of DG-modules and the $R$-linear chain maps between DG-modules form an abelian category denoted by $\textsf{DG}(R)$. By formally inverting the quasi-isomorphisms in this category, we obtain the derived category of $R$, denoted by $\D(R)$. If $M\in\D(R)$ then we set
\[
\inf X=\inf\{i\mid \H_i(M)\neq0\},\quad\sup M=\sup\{i\mid \H_i(M)\neq0\},
\]
and we set $\amp M=\sup M-\inf M$. Furthermore, we define
\[
\Dba[R]=\{M\in \D(R)\mid \sup M<\infty\},\quad \Dbb[R]=\{M\in\D(R)\mid \inf M>-\infty\},
\]
and $\Db[R]=\Dba[R]\cap\Dbb[R]$. Moreover, we say that $M\in\Df[R]$ if $\H_i(M)$ is a finitely generated $\HR$-module for all $i$. We also set
\begin{itemize}
\item $\Dfba[R]=\Df[R]\cap\Dba[R], $
\item $\Dfbb[R]=\Df[R]\cap\Dbb[R],$
\item $\Dfb[R]=\Df[R]\cap\Db[R].$
\end{itemize}
We denote an isomorphism in $\D(R)$ by $\simeq$.

\subsection*{Homological Dimensions}
For $M \in \D(R)$, the projective, injective, and flat dimensions of $M$ are defined similarly to the classical homological dimensions in commutative algebra. The projective dimension of $M$ is defined to be
\begin{equation*}
	\projdim_RM = \inf\{n \in \mathbb{Z} \, | \, \Ext^i_R(M, N) = 0 \text{ for any $N \in \Db[R]$ and any $i > n - \inf N$}\},
\end{equation*}
the injective dimension of $M$ is defined to be
\begin{equation*}
    \injdim_RM = \inf\{n \in \mathbb{Z} \, | \, \Ext_i^R(N, M) = 0 \text{ for any $N \in \Db[R]$ and any $i > n + \sup N$}\},
\end{equation*}
and the flat dimension of $M$ is defined to be
\begin{equation*}
	\flatdim_R M = \inf\{n \in \mathbb{Z} \, | \, \Tor_i^R(N, M) = 0 \text{ for any $N \in \Db[R]$ and any $i > n + \sup N$}\}.
\end{equation*}

\subsection*{Depth and local cohomology Krull dimension} If $X\in\Dbb[R]$, then its \emph{local cohomology Krull dimension} is
\[
\ld_RX=\sup_{l\in\mathbb{Z}}\{\dim_{\HR}\H_l(X)-l\},
\]
see \cite[Definition 2.1]{ShaulCM}.

If $X\in\Dba[R]$, then its \emph{depth} is
\[
\depth_RX=-\sup\RHom_R(\kk,X),
\]
see \cite[Definition 3.1]{ShaulCM}. If $I$ is an ideal of $\H_0(R)$, we define the \emph{$I$-depth} of $X$ as
\[
\depth_R(I,X)=-\sup\RHom_R(\H_0(R)/I,X),
\]
see \cite[Definition 3.1]{ShaulKoszul}.

\subsection*{Local (co)homology} Let $I$ be an ideal of $\HR$. Let $J$ be a finitely generated ideal in $R_0$ whose image in $\HR$ is $I$ and let $\mathbf{j}$ be a finite generating set of $J$. We denote by $\check{C}(\mathbf{j})$ the \v{C}ech complex on the sequence $\mathbf{j}$, see \cite[(4.13)]{ShaulYeku}, where it is called the \emph{infinite dual Koszul complex} on $\mathbf{j}$. If $X\in\D(R)$ we define the \emph{derived $I$-torsion} and the \emph{derived $I$-completion} of $X$ as
\[
\RG[\m]{R}{X}=X\lotimes_{R_0}\check{C}(\mathbf{j}),\quad\LL[\m]{R}{X}=\RHom_{R_0}(\check{C}(\mathbf{j}),X)
\]
see \cite[Lemma 5.7]{ShaulYeku} and \cite[Corollary 2.13]{ShaulCompletion}. We omit the dependency on the ring if it is clear from the context.

We say that $X$ is \emph{derived $I$-torsion} if $\RGam[\m]{X}\simeq X$, and we say that it is \emph{derived $I$-complete} if $\LLam[\m]{X}\simeq X$.

\section{The Derived Improved New Intersection Theorem}\label{sec:dinit}

Over commutative noetherian local rings the New Intersection Theorem, originally proved by Peskine and Szpiro \cite{PeskineSzpiro2}, is concerned with the length of finite free complexes. The Improved New Intersection Theorem first appeared in the proof of Evans and Griffith's Syzygy Theorem \cite{EvansGriffith}. Two stronger versions were later proved by Iyengar \cite[Theorem 3.1]{Sri} and by Avramov, Iyengar and Neeman \cite[Theorem 4.2]{AIN}. More recently, the first author and Christensen proved a version of the Improved New Intersection Theorem \cite[Theorem 2.2]{IntThm} which subsumes both Iyengar's version and the version of Avramov, Iyengar and Neeman.

In a recent paper, the second author generalized Iyengar's intersection theorem to DG-rings, see \cite[Theorem 5.6]{Zach}, labeling it the Derived Improved New Intersection Theorem. The main result of this section, \Cref{thm:dinit}, is a DG version of the intersection theorem proved by the first author and Christensen and shows that the Derived Improved New Intersection Theorem holds under weaker hypotheses.

We first prove a series of five preliminary lemmas which are DG versions of \cite[Proposition 13.3.19, Theorem 14.4.3]{LarsBook}, \cite[Lemma 1.3(b)]{IntThm}, \cite[Proposition 2.10]{FoxbyIyengar} and \cite[(2.2)]{MCMSurvey}.

\begin{lemma}\label{lem:RGamTensor}
Let $R$ be a DG-ring and let $I$ be an ideal of $\H_0(R)$. Let $M,N\in\D(R)$, then
\[
(\RGam[I]{M})\lotimes_RN\simeq \RGam[I]{(M\lotimes_RN)}\simeq M\lotimes_R(\RGam[I]{N}).
\]

\end{lemma}
\begin{proof}
Let $J$ be a finitely generated ideal in $R_0$ whose image in $\H_0(R)$ is $I$. Let $\mathbf{j}$ be a finite generating set of $J$. Let $\check{C}(\mathbf{j})$ be the \v{C}ech complex on the sequence $\mathbf{j}$. Then by \cite[Proposition 2.4]{ShaulCompletion} and \cite[Proposition 12.10.9]{Yeku}
\[
\RGam[I]{M}\lotimes_RN\simeq (M\lotimes_{R_0}\check{C}(\mathbf{j}))\lotimes_RN\simeq (M\lotimes_RN)\lotimes_{R_0}\check{C}(\mathbf{j})\simeq \RGam[I]{(M\lotimes_RN)}.\qedhere
\]
\end{proof}

The following Lemma generalizes \cite[Proposition 3.3]{ShaulCM}.
\begin{lemma}\label{lem:Idepth}
Let $R$ be a DG-ring and $I$ an ideal of $\H_0(R)$. Let $N\in\Dba[R]$ be nonzero, then
\[
\depth_R(I,N)=-\sup\RGam[I]{N}.
\]
\end{lemma}
\begin{proof}
The assertion follows from the following string of equalities where $K(R;x_1,\ldots,x_n)$ denotes the Koszul complex over a generating set of $I$ as defined in \cite[Definition 2.4]{ShaulKoszul}
\begin{align*}
\depth_R(I,N)&=n-\sup (K(R;x_1,\ldots,x_n)\lotimes_RN)\\
&=n-\sup (\RGam[I]K(R;x_1,\ldots,x_n)\lotimes_RN)\\
&=n-\sup (K(R;x_1,\ldots,x_n)\lotimes_R\RGam[I](N))\\
&=\depth_R(I,\RGam[I]{N})\\
&=-\sup\RGam[I]{N},
\end{align*}
where the first equality follows from \cite[Proposition 3.17]{ShaulKoszul}, and the second equality from \cite[Proposition 2.9]{ShaulCompletion} and the fact that $K(R;x_1, \ldots, x_n)$ is derived $I$-torsion by construction. The third equality from \Cref{lem:RGamTensor}, the fourth is again \cite[Proposition 3.17]{ShaulKoszul}, and the last one from \cite[Lemma 3.1]{Zach}.
\end{proof}

\begin{lemma}\label{lem:supRHom2}
Let $R$ be a DG-ring and $I$ an ideal of $\H_0(R)$. Let $M\in\Dbb[R]$ and $N\in\Dba[R]$ be DG $R$-modules such that $M$ is derived $I$-torsion with $\H(M)$ nonzero. Then
\[
-\sup\RHom_R(M,N)\geq\inf M+\depth_R(I,N).
\]
\end{lemma}

\begin{proof}
We have that the following chain of isomorphisms in the derived category holds:

\begin{align*}
    \RHom_R(M, N)   &\simeq \RHom_R(\RGam[I]{M}, N) \\
                    &\simeq \RHom_R(\RGam[I]{M}, \RGam[I]{N}) \\
                    &\simeq \RHom_R(M, \RGam[I]{N}).
\end{align*}

The first isomorphism and third isomorphisms hold as $M$ is derived $I$-torsion, and the second isomorphism holds by derived Greenlees-May duality \cite[Proposition 2.11]{ShaulCompletion}. This gives the first equality below:

\begin{align*}
    -\sup\RHom_R(M, N))   &= -\sup\RHom_R(M, \RGam[I]{N})\\
                            &\geq \inf M - \sup\RGam[I]{N} \\
                            &= \inf M + \depth_R(I, N).
\end{align*}
where the inequality comes from \cite[Lemma 3.2(i)]{ShaulFinDim} and the last equality from \Cref{lem:Idepth}.
\end{proof}

Below we use the following notation: if $I$ is an ideal of $\HR$ then
\[
\mathrm{V}(I)=\{\p\in\Spec(\HR)\mid I\subseteq\p\}.
\]

\begin{lemma}\label{lem:depthLocal}
Let $R$ be a DG-ring and $M$ a DG $R$-module in $\Dba[R]$. If $I$ is an ideal of $\H_0(R)$, then
\[
\depth_R(I,M)=\inf\{\depth_{R_\p}M_\p\mid\p\in\mathrm{V}(I)\}.
\]
\end{lemma}

\begin{proof}
It follows from derived tensor-hom adjunction that
\[
\RHom_{\H_0(R)}(\H_0(R)/I,\RHom_R(\H_0(R),M))\simeq\RHom_R(\H_0(R)/I,M).
\]
Therefore the second equality below follows
\begin{align*}
\depth_R(I,M)&=-\sup\RHom_R(\H_0(R)/I,M)\\
&=-\sup\RHom_{\H_0(R)}(\H_0(R)/I,\RHom_R(\H_0(R),M)).
\end{align*}
Therefore
\begin{align*}
\depth_R(I,M)&=\depth_{\H_0(R)}(I,\RHom_R(\H_0(R),M))\\
&=\inf\{\depth_{\H_0(R)_\p}\RHom_R(\H_0(R),M)_\p\mid\p\in\mathrm{V}(I)\}\\
&=\inf\{\depth_{R_\p}M_\p\mid\p\in\mathrm{V}(I)\},
\end{align*}
where the second equality follows from \cite[Proposition 2.10]{FoxbyIyengar}.
\end{proof}

\begin{lemma}\label{lem:dCompleteFree}
Let $R$ be a DG-ring with $I$ an ideal in $\H_0(R)$, and let $M\in \D(R)$ and  $F\in\Dfb$ with $\projdim_RF < \infty$. Then
\[
\LLam[I]{(M \lotimes_R F)} \simeq \LLam[I]{(M)} \lotimes_R F
\]
\end{lemma}
\begin{proof}
Let $J$ be a finitely generated ideal in $R_0$ that is a lift of $I$ (that is, let $J$ be the preimage of $I$ under the quotient map $R_0 \to \H_0(R)$), and denote a finite list of generators of $J$ as $\mathbf{j}$. We have an isomorphism $\LLam[I]{(M \lotimes_R F)} \simeq \RHom_{R_0}(\check{C}(\mathbf{j}), M \lotimes_R F)$ in $\D(R)$ by \cite[Lemma 2.4]{ShaulCompletion} The DG $R$-module $\RHom_{R_0}(\check{C}(\mathbf{j}), M \lotimes_R F)$ is isomorphic to $\RHom_R(\check{C}(\mathbf{j}) \lotimes_{R_0} R, M \lotimes_R F)$ through derived Hom-tensor adjunction. Thus, we have the isomorphisms

\begin{align*}
\LLam[I]{(M \lotimes_R F)}  &\simeq \RHom_R(\check{C}(\mathbf{j}) \lotimes_{R_0} R, M \lotimes_R F) \\
                            &\simeq \RHom_R(\check{C}(\mathbf{j}) \lotimes_{R_0} R, M) \lotimes_R F \\
                            &\simeq \LLam[I]{M} \lotimes_R F
\end{align*}
where the second isomorphism comes from \cite[Lemma 2.7(ii)]{ShaulFinDim}, and the third isomorphism comes from \cite[Lemma 2.4]{ShaulCompletion}.
\end{proof}

The following Theorem improves the bound given in \cite[Lemma 3.7]{Zach}.
\begin{theorem}\label{thm:depthIneq}
Let $R$ be a DG-ring and $M$ a derived $\m$-complete complex. For every prime ideal $\p$ in $\H_0(R)$ the following inequality holds
\[
\depth_RM\leq\depth_{R_\p}M_\p+\dim\H_0(R)/\p.
\]
\end{theorem}

\begin{proof}
The claim follows from the following chain of (in)equalities
\begin{align*}
\depth_{R_\p}M_\p&\geq\depth_R(\p,M)\\
&=-\sup\RHom_R(\H_0(R)/\p,M)\\
&=-\sup\RHom_R(\H_0(R)/\p,\LLam{M})\\
&=-\sup\RHom_R(\RGam{\H_0(R/\p)},M)\\
&\geq\inf\RGam{\H_0(R)/\p}+\depth_RM\\
&=\depth_RM-\dim\H_0(R)/\p,
\end{align*}
where the first inequality follows from \cite[Proposition 3.3]{ShaulKoszul}, the first equality follows by definition, the second inequality follows since $M$ is derived $\m$-complete, the third equality follows from \cite[Proposition 2.11]{ShaulCompletion} and \cite[Remark 2.14]{ShaulCompletion}, the second inequality follows from \Cref{lem:supRHom2}, and the last equality follows from \cite[Theorem 2.15]{ShaulCM}.
\end{proof}

\begin{corollary}\label{cor:depth}
Let $R$ be a DG-ring and $M$ a derived $\m$-complete $R$-complex. For every ideal $I$ of $\H_0(R)$ the following inequality holds
\[
\depth_RM\leq\depth_R(I,M)+\dim\H_0(R)/I.
\]
\end{corollary}

\begin{proof}
By \Cref{lem:depthLocal} $\depth_R(I,M)=\depth_{R_\p}M_\p$ holds for some $\p\in \mathrm{V}(I)$. Since $\dim \H_0(R)/\p\leq\dim\H_0(R)/I$ holds, \Cref{thm:depthIneq} gives the desired inequality.
\end{proof}

We are now ready to prove the main result of this section, which shows that \cite[Theorem 5.6]{Zach} holds under weaker hypotheses. We recall that a DG-ring $R$ is said to have \emph{constant amplitude} if
\[
\amp R_\p=\amp R,\text{ for all }\p\in\Spec(\HR).
\]

\begin{theorem}\label{thm:dinit}
Let $R$ be a DG-ring with constant amplitude. Let $I$ be an ideal of $\H_0(R)$, and $F\in\Dfb[R]$ with $\inf F=0$. If $\H_i(F)$ is $I$-torsion for $i>0$ and a minimal generator of $\H_0(F)$ is $I$-torsion, then
\[
\projdim_RF+\amp R\geq\dim\H_0(R)-\dim\H_0(R)/I.
\]
\end{theorem}

\begin{proof}
Without loss of generality assume $\inf R=0$ and $\projdim_RF<\infty$. As in \cite[Theorem 5.6]{Zach} we can assume that there is a DG $R$-module of maximal depth $M$ (see \cite[Definition 3.2]{Zach}) by passing to the derived $\m$-adic completion of $R$. This DG $R$-module of maximal depth can be assumed to be derived $\m$-complete by \cite[Lemma 3.6]{Zach}. Let $s=\sup F\otimes_RM$, we notice that by \cite[Lemma 3.5]{Zach} $s\geq-\amp R$.

Let $\p\in\Ass_{\H_0(R)}\H_s(F\otimes_RM)$. It follows that $\H(F\otimes_RM)_\p$ is nonzero and therefore so are $\H(F)_\p$ and $\H(M)_\p$. One has the following chain of (in)equalities
\begin{align*}
\projdim_{R_\p}F_\p&=\depth_{R_\p}M_\p-\depth_{R_\p}(F\otimes_RM)_\p\\
&=\depth_{R_\p}M_{\p}+s\\
&\geq\depth_RM-\dim\H_0(R)/\p+s\\
&=\dim\H_0(R)-\dim\H_0(R)/\p+s,
\end{align*}
where the first equality is the Auslander-Buchsbaum formula \cite[Theorem 3.2]{Zach}, the second equality follows from \cite[Lemma 3.1]{Zach}, the inequality holds by \Cref{thm:depthIneq}, and the last equality holds as $M$ has maximal depth.

Assume first that $s>-\amp R$. It suffices to show that $I\subseteq\p$ since in this case one has
\begin{align*}
\projdim_RF&\geq\projdim_{R_\p}F_\p\\
&\geq\dim\H_0(R)-\dim\H_0(R)/\p+s\\
&>\dim \H_0(R)-\dim\H_0(R)/I-\amp R.
\end{align*}
To prove that $I\subseteq\p$ we assume that $I\not\subseteq\p$ and seek a contradiction. It follows that $F_\p$ is isomorphic to $\H_0(F)_\p$ in $\D(R)$ and therefore $\sup F_\p=0$. Now one has the following chain of (in)equalities
\begin{align*}
\depth_{R_\p} R_\p&=\depth F_\p+\projdim_{R_\p}F_\p\\
&\geq-\sup F_\p+\projdim_{R_\p}F_\p\\
&\geq\dim \H_0(R)-\dim\H_0(R)/\p+s\\
&\geq\dim \H_0(R)_\p+s\\
&>\dim\H_0(R)_\p-\amp R
\end{align*}
where the first equality is the Auslander-Buchsbaum formula \cite[Theorem 3.2]{Zach}, the first inequality from \cite[Proposition 3.2]{ShaulKoszul}, the second was proved above, and the third inequality is standard. By \cite[Proposition 3.5]{ShaulCM} the first inequality below holds
\begin{align*}
\depth_{R_\p}R_\p&\leq\dim\H_{\sup R_\p}(R_\p) - \sup R_\p\\
&\leq \dim\H_0(R_\p) - \sup R_\p \\
&=\dim \H_0(R)_\p-\amp R.
\end{align*}
The second inequality holds as $\H_{\sup R_\p}(R_\p)$ is an $\H_0(R_\p)$ module, and the last equality follows since $R$ has constant amplitude. This is a contradiction.

It remains to consider the case $s=-\amp R$. Since $\H_0(F)$ is finitely generated, every minimal generator in $\H_0(F)$ gives rise to a minimal generator in $\H_0(F)\otimes_{\H_0(R)}\kk$ by Nakayama's Lemma. By \cite[Lemma 3.2]{ShaulFinDim} $\H_0(F)\otimes_{\H_0(R)}\kk\simeq\H_0(F\lotimes_R\kk)$ as $\inf F=0$. The minimal generator in $\H_0(F\lotimes_R\kk)$ corresponding to the minimal generator of $\H_0(F)$ that is $I$-torsion gives rise to a nonzero element in $\H_{-\amp R}(F\lotimes_R(\kk\otimes_RM))$ by \cite[Diagram 3.3.2 in Theorem 3.5]{Zach}. The commutativity of \cite[Diagram 3.3.2]{Zach} implies that there is a nonzero element in $\H_{-\amp R}(M\otimes_RF)$ that is $I$-torsion, and so $\Gamma_I\H_{-\amp R}(M\otimes_RF)\neq0$. Since $\sup(M\otimes_RF)=-\amp R$, this implies that $\depth_R(I,M\otimes_RF)=\amp R$ by \cite[Lemma 3.1]{Zach}. By \Cref{lem:dCompleteFree} $F\otimes_RM$ is derived $\m$-complete and therefore \Cref{cor:depth} yields
\[
\depth_R(F\otimes_RM)\leq\depth_R(I,M)+\dim\H_0(R)/I=\amp R+\dim\H_0(R)/I.
\]
Applying the Auslander-Buchsbaum formula \cite[Theorem 3.2]{Zach} one has
\begin{align*}
\projdim_RF&=\depth_RM-\depth_R(F\otimes_RM)\\
&=\dim\H_0(R)-\depth_R(F\otimes_RM)\\
&\geq\dim\H_0(R)-\amp R-\dim\H_0(R)/I.\qedhere
\end{align*}
\end{proof}

\section{Foxby's Intersection Theorem}\label{sec:foxby}

Over a commutative noetherian local ring $R$, the Intersection Theorem, first proved by Peskine and Szpiro \cite{PeskineSzpiro} in the equicharacteristic case and later by Roberts \cite{Roberts1,Roberts2} in the general case, states that if $X$ and $Y$ are finitely generated modules with $X$ of finite projective dimension, then
\[
\projdim_RX+\dim_R X\otimes_RY\geq\dim_RY.
\]
This result was later generalized by Foxby \cite{foxflat} for complexes: if $X\in\Db[R]$ with finite flat dimension and $Y\in\Dfb[R]$, then
\[
\flatdim_RX+\dim_R X\lotimes_RY\geq\dim_RY.
\]
More recently, Yang generalized Foxby's result over DG-rings, see \cite[Theorem 3.1]{yang}. Namely, if $R$ is a DG-ring, $X\in\Db[R]$ of finite flat dimension and finite depth and $0\not\simeq Y\in\Dfbb[R]$ , then
\begin{equation}\label{eq:Yang}
\flatdim_RX+\ld_R X\lotimes_RY\geq\ld_RY.
\end{equation}
The main result of this section, \Cref{thm:foxby2}, improves Yang's result by replacing $\flatdim_RX$ with the smaller quantity $\depth R-\depth_RX$ (see \Cref{cor:flat>}) and replacing $\ld_RX\lotimes_RY$ with the smaller quantity $-\inf\RG[\m]{R}{(X\lotimes_RY)}$ (see \Cref{rem:infRGamma}). We start by proving some preliminary results.

\begin{proposition}\label{prp:abFlat} 
Let $X\in\Db[R]$ with $\flatdim_RX<\infty$, and let $Y\in\Dba[R]$ be derived $\m$-torsion, then
\[
\sup X\lotimes_RY-\sup Y=\depth R-\depth X.
\]
\end{proposition}
\begin{proof}
It follows from \cite[Lemma 2.5]{yang} that
\begin{equation}\label{eq:ddf}
\depth X\lotimes_RY=\depth X+\depth Y-\depth R.
\end{equation}
Moreover by \Cref{lem:RGamTensor} $X\lotimes_RY$ is also derived $\m$-torsion. Now the desired equality follows from \eqref{eq:ddf} and from the observation that by \Cref{lem:Idepth} one has
\[
\depth X\lotimes_RY=-\sup X\lotimes_RY,\quad \depth Y=-\sup Y.\qedhere
\]
\end{proof}

\begin{corollary}\label{cor:abFlat}
If $X\in\Db[R]$ and $\flatdim_RX<\infty$, then
\[
\sup \kk\lotimes_RX=\depth R-\depth X.
\]
\end{corollary}
\begin{proof}
Follows immediately from \Cref{prp:abFlat} since $\kk$ is derived $\m$-torsion.
\end{proof}

The next corollary allows us to link the depth of a DG $R$-module $X$ of finite flat dimension to the depth over $\HR$ of $X\lotimes_R\HR$. This will be used in subsequent proofs which work by reduction to $\HR$.

\begin{corollary}\label{cor:depthBaseChange}
Let $X\in\Db[R]$ with $\flatdim_RX<\infty$, then
\[
\depth_RX=\depth_{\H_0(R)} X\lotimes_R\H_0(R)+\depth R-\depth H_0(R).
\]
\end{corollary}
\begin{proof}
By \Cref{cor:abFlat}
\[
\depth_RX=\depth R-\sup\kk\lotimes_RX.
\]
By \cite[Corollary 16.3.3]{LarsBook} and \cite[Lemma 2.7]{ShaulFinDim} it follows that
\[
\depth_{\H_0(R)}X\lotimes_R\H_0(R)=\depth \H_0(R)-\sup\kk\lotimes_{\H_0(R)}X\lotimes_R\H_0(R),
\]
now one concludes by tensor cancellation.
\end{proof}

The result below shows that for bounded DG $R$-modules $X$ of finite flat dimension $\depth R-\depth_RX$ is smaller than the flat dimension of $X$.
\begin{corollary}\label{cor:flat>}
Let $X\in\Db[R]$ with $\flatdim_RX<\infty$, then
\[
\flatdim_RX\geq\depth R-\depth X.
\]
\end{corollary}
\begin{proof}
The asserted inequality follows from the following string of (in)equalities
\begin{align*}
\flatdim_RX&=\flatdim_{H_0(R)}X\lotimes_R\H_0(R)\\
&\geq\depth H_0(R)-\depth_{H_0(R)}X\lotimes_R\H_0(R)\\
&=\depth H_0(R)-\depth_RX+\depth R-\depth H_0(R)\\
&=\depth R-\depth X.
\end{align*}
The first equality follows from \cite[Corollary 2.3]{ShaulFinDim}, the inequality follows from \cite[Corollary 17.3.4]{LarsBook}, the second equality from \Cref{cor:depthBaseChange}.
\end{proof}

\begin{remark}\label{rem:infRGamma}
We recall that by \cite[Proposition 2.9]{ShaulCM} if $M\in\Dba[R]$, then $-\inf\RG[\m]{R}{M}\leq\ld_RM$.
\end{remark}

The theorem below is our improvement of Yang's result.

\begin{theorem}\label{thm:foxby2}
Let $X\in\Db[R]$ and $Y\in\Dfbb[R]$ with $Y\not\simeq0$, $\flatdim_RX<\infty$ and $\depth_RX<\infty$. Then,
\[
\depth R-\depth_RX-\inf\RG[\m]{R}{(X\lotimes_RY)}\geq\ld_RY.
\]
\end{theorem}
\begin{proof}
The asserted inequality follows from the following string of (in)equalities
\begin{align*}
-\inf\RG[\m]{R}{(X\lotimes_RY)}&=-\inf\RG[\m]{\HR}{(X\lotimes_RY\lotimes_R\HR)}\\
&=-\inf\RG[\m]{\HR}{((X\lotimes_R\HR)\lotimes_{\HR}(Y\lotimes_R\HR))}\\
&=\dim_{\HR}\RG[\m]{\HR}{((X\lotimes_R\HR)\lotimes_{\HR}(Y\lotimes_R\HR))}\\
&=\dim_{\HR}\RG[\m]{\HR}{(X\lotimes_R\HR)}\lotimes_{\HR}(Y\lotimes_R\HR)\\
&\geq\depth_{\HR}\RG[\m]{\HR}{(X\lotimes_R\HR)}+\dim_{\HR}Y\lotimes_R\HR-\depth\HR\\
&=\depth_{\HR}(X\lotimes_R\HR)+\dim_{\HR}Y\lotimes_R\HR-\depth\HR\\
&=\depth_RX+\ld_RY-\depth R.
\end{align*}
Where the first equality follows from \cite[Lemma 2.12]{ShaulCM}, the second is just tensor cancellation, the third comes from \cite[Theorem 16.1.34(b)]{LarsBook}, the fourth from \cite[Proposition 13.3.19]{LarsBook}, the inequality from \cite[Proposition 18.5.1]{LarsBook} (which can be applied by \cite[Corollary 2.3]{ShaulFinDim}), the fifth equality follows from \cite[Theorem 16.2.14]{LarsBook}, and the last equality from \cite[Proposition 2.13]{ShaulCM} and \Cref{cor:depthBaseChange}.
\end{proof}

As a corollary we deduce the following amplitude inequality for the local cohomology of a derived tensor product, which generalizes the right inequality of \cite[Theorem 3.1]{amp} to DG-rings under some slightly stronger hypotheses.

\begin{corollary}\label{cor:ampRGam}
Let $X\in\Db[R]$ and $Y\in\Dfb[R]$ with $Y\not\simeq0$, $\flatdim_RX<\infty$ and $\depth_RX<\infty$. Then,
\[
\amp\RGam[\m]{(X\lotimes_RY)}\geq\amp\RGam[\m]{Y}.
\]
\end{corollary}
\begin{proof}
The desired inequality follows form the following chain of (in)equalities
\begin{align*}
\amp\RGam[\m]{(X\lotimes_RY)}&=\sup\RGam[\m]{(X\lotimes_RY)}-\inf\RGam[\m]{(X\lotimes_RY)}\\
&=-\depth_R(X\lotimes_RY)-\inf\RGam[\m]{(X\lotimes_RY)}\\
&=-\depth_RX-\depth_RY+\depth R-\inf\RGam[\m]{(X\lotimes_RY)}\\
&\geq-\depth_R X-\depth_R Y+\depth R+\ld_R Y-\depth R+\depth_RX\\
&=\ld_R Y-\depth_R Y\\
&=\amp\RGam[\m]{Y}.
\end{align*}
Where the second equality follows from \Cref{lem:Idepth}, the third from \cite[Lemma 2.5]{yang}, the inequality from \Cref{thm:foxby2}, and the last equality from \cite[Proposition 2.8 and Proposition 3.3]{ShaulCM}.
\end{proof}
It is known that over commutative noetherian local rings the existence of a finite length module of finite projective dimension forces the ring to be Cohen-Macaulay. This result was generalized to DG-rings by Yang, see \cite[Corollary 4.4]{yang}. The final result of this section generalizes Yang's result by replacing the finite projective dimension hypothesis with a finite flat dimension hypothesis. We recall that $R$ is said to be \emph{Cohen-Macaulay} if $\amp R=\amp\RGam[\m]{R}$, see \cite[Definition 4.2]{ShaulCM}.

\begin{corollary}\label{cor:fl}
The following are equivalent
\begin{enumerate}
\item $R$ is Cohen-Macaulay.
\item There exists a derived $\m$-torsion DG-module $X\in\Db[R]$ with $\flatdim_RX<\infty$ and $\amp X=\amp R$.
\end{enumerate}
\end{corollary}
\begin{proof}
If $R$ is Cohen-Macaulay, then $X=\RGam[\m]{R}$ satisfies the required properties. Now assume that such an $X$ exists. It follows from \Cref{cor:ampRGam} that $\amp R=\amp X\geq\amp\RGam[\m]{R}$, while by \cite[Theorem 4.1]{ShaulCM} one has $\amp R\leq\amp\RGam[\m]{R}$. Therefore $\amp R=\amp\RGam[\m]{R}$, i.e. $R$ is Cohen-Macaulay.
\end{proof}

Next we show that the inequality in \Cref{thm:foxby2} is stronger than the one provided in \cite[Theorem 3.1]{yang}, but first we need some preliminary results. In particular, we need a version of \cite[Proposition 3.6]{ShaulFinDim} for flat dimension.

\begin{lemma}\label{lem:flatQuotient}
Let $M\in \Db[R]$. There is an equality
\begin{equation*}
\flatdim_R M = \sup\{n \in \mathbb{Z} \, | \, \Tor^R_n(\HR/I, M) \neq 0 \text{ for some ideal } I \subseteq \HR \}.
\end{equation*}
\end{lemma}
\begin{proof}
By \cite[Corollary 2.3(ii)]{ShaulFinDim}, we have that $\flatdim_R M = \flatdim_{\HR}\HR \lotimes_R M$. By \cite[Theorem 8.3.11]{LarsBook}, we have that
\begin{equation*}
\flatdim_{\HR}\HR \lotimes_R M =  \sup\{n \in \mathbb{Z} \, | \, \Tor^{\HR}_n(\HR/I, \HR \lotimes_R M) \neq 0 \text{ for some ideal } I \subseteq \HR\}.
\end{equation*}
The equality then follows by derived tensor cancellation.
\end{proof}

\begin{proposition}\label{prp:injdim}
Let $Y\in\Dba[R]$ with $\injdim_RY<\infty$ and $X\in\D(R)$, then
\[
\flatdim_R\RHom_R(X,Y)\leq\sup Y+\injdim X
\]
\end{proposition}
\begin{proof}
Without loss of generality we can assume $\injdim_RX<\infty$. Consider $\HR/I$ for any ideal $I \subseteq \HR$. By Hom evaluation \cite[Lemma 2.7]{ShaulFinDim}, we have that
\[
\HR/I\lotimes_R\RHom_R(X,Y)\cong\RHom_R(\RHom_R(\HR/I,X),Y).
\]
By \cite[Lemma 3.2]{ShaulFinDim} one has the first inequality below, while the second follows from \cite[Lemma 3.5]{ShaulFinDim}
\begin{align*}
\sup\RHom_R(\RHom_R(\HR/I,X),Y)&\leq\sup Y-\inf\RHom_R(\HR/I,X)\\
&\leq\sup Y + \injdim X.
\end{align*}
By \Cref{lem:flatQuotient} this yields the desired inequality.
\end{proof}
We recall that $R$ is said to be \emph{Gorenstein} if $\injdim_RR < \infty$ (see \cite{GorensteinDGA},\cite{ShaulCM}). Note that by \cite[Proposition 4.5]{ShaulCM}, a Gorenstein DG-ring is Cohen-Macaulay.  
\begin{corollary}\label{cor:injdim}
Let $R$ be a Gorenstein DG-ring and $X\in\Db[R]$. Then,
\[
\injdim_RX<\infty\Longleftrightarrow\flatdim_RX<\infty.
\]
\end{corollary}
\begin{proof}
If $\injdim_RX<\infty$, then by \Cref{prp:injdim}
\begin{align*}
\flatdim_R X&=\flatdim_R\RHom_R(R,X)\\
&\leq\sup X+\injdim R.
\end{align*}
If $\flatdim_RX<\infty$, then by \cite[Proposition 3.6]{ShaulFinDim}
\begin{align*}
\injdim_RX&=\injdim_R R\lotimes_RX\\
&\leq \injdim R-\inf X.\qedhere
\end{align*}
\end{proof}

\begin{example}
In this example we construct a DG-ring $R$ of arbitrarily large amplitude and a bounded complex $X$ of finite depth and finite flat dimension such that
\[
\ld_RX>-\inf\RGam[\m]{X},
\]
showing that the inequality in \Cref{thm:foxby2} is indeed stronger than the one provided in \cite[Theorem 3.1]{yang}.

Let $(A,\m)$ be a Gorenstein local noetherian ring of positive Krull dimension. Let $R$ be the trivial extension $A\ltimes A[n]$ for some positive integer $n$ as defined in \cite[Definition 1.2]{trivialExt}. By \cite[Theorem 2.2]{trivialExt} $R$ is a Gorenstein DG-ring. The homology of $R$ is
\[
\H_i(R)=\begin{cases} A\quad i=0,n\\0 \quad\;\mathrm{otherwise,}\end{cases}
\]
therefore $\amp R=n$. Let $\p\in\Spec A$, we denote by $E(R,\p)$ the indecomposable injective DG $R$-module such that $\H_0(E(R,\p))\cong E_A(A/\p)$, the injective hull of $A/\p$ over $A$, see \cite[7.2]{injective}. Let $\p$ be a minimal prime of $A$ and consider the complex $X=E(R,\p)\oplus E(R,\m)$. We note that by \cite[Corollary 4.12]{injective}
\begin{align*}
\H_i(X)&\cong\H_i(E(R,\p))\oplus\H_i(E(R,\m))\\
&\cong\Hom_A(\H_{-i}(R),E_A(A/\p))\oplus\Hom_A(\H_{-i}(R),E_A(A/\m))\\
&\cong\begin{cases} E_A(A/\p)\oplus E_A(A/\m)\quad i=0,-n\\
0\quad\;\mathrm{otherwise,}\end{cases}
\end{align*}
in particular $X$ is bounded. Moreover,
\begin{align*}
\depth_R X&=\inf\{\depth_R E(R,\p),\depth_R E(R,\m)\}\\
&=\inf\{-\sup\RGam[\m]{E(R,\p)},-\sup\RGam[\m]{E(R,\m)}\}\\
&=-\sup E(R,\m)\\
&=0,
\end{align*}
where the third equality is from \cite[Proposition 7.13]{injective} and the last one from \cite[Definition 3.1]{injective}.

We note that by \Cref{cor:injdim} $\flatdim_RX<\infty$.
Furthermore,
\begin{align*}
\ld_RX&=\sup_{l\in\mathbb{Z}}\{\dim_A\H_l(X)-l\}\\
&=\dim_A(E_A(A/\p)\oplus E_A(A/\m))+n\\
&=\dim A+n.
\end{align*}

Finally,
\begin{align*}
-\inf\RGam[\m]{X}&=-\inf(\RGam[\m]{E(R,\p)}\oplus\RGam[\m]{E(R,\m)})\\
&=-\inf E(R,\m)\\
&=n,
\end{align*}
where the second equality follows from \cite[Proposition 7.13]{injective}, and the third follows similarly as the computation of the homology of $X$. 
\end{example}

\bibliographystyle{amsplain}
\bibliography{biblio}
\end{document}